\documentclass[final,reqno]{amsart}

\usepackage{graphicx}
\usepackage{amsmath,amsthm,amsfonts,amscd,amssymb}
\usepackage[color,notref,notcite]{showkeys}
\usepackage{url}
\usepackage{subeqnarray}
\usepackage[color,notref,notcite]{showkeys} 
\definecolor{labelkey}{rgb}{0.6,0,1}
\usepackage{enumitem}
\usepackage{algorithm}
\usepackage{algpseudocode}
\usepackage{citesort}
\usepackage{esint}

\theoremstyle{plain}
\newtheorem{theorem}{Theorem}[section]

\newtheorem{lemma}[theorem]{Lemma}

\newtheorem{hypothesis}[theorem]{Hypothesis}

\theoremstyle{definition}
\newtheorem{definition}[theorem]{Definition}

\def\bhyp#1{\begin{equation}\label{#1}\begin{array}{c}}
\def\ehyp{\end{array}\end{equation}}

\newcounter{cst}

\theoremstyle{remark}
\newtheorem{remark}[theorem]{Remark}

\numberwithin{equation}{section}
\numberwithin{figure}{section}

\newcommand{\RR}{{\mathbb R}}
\newcommand{\NN}{{\mathbb N}}

\def\O{\Omega}
\def\dsp{\displaystyle}

\def\disc{{\mathcal D}}
\def\mesh{{\mathcal M}}
\newcommand{\polyd}{{\mathcal T}}
\def\edges{{\mathcal E}}
\def\edge{\sigma}

\newcommand{\edgesext}{{{\edges}_{\rm ext}}} 

\def\dr{\partial}
\newcommand{\centeredge}{\overline{x}_\edge} 

\newcommand{\cK}{{\mathcal K}}

\DeclareMathOperator*{\argminB}{argmin}

\newcommand{\x}{\pmb{x}}

\newif\ifcorr\corrtrue
\usepackage[normalem]{ulem}
\definecolor{violet}{rgb}{0.580,0.,0.827}

\def\bpsi{{\boldsymbol \psi}}

\newcommand{\ud}{\, \mathrm{d}} 
\def\div{\mathop{\rm div}}
\def\ini{\mathop{\rm ini}}

\title{Convergence Analysis For Non Linear System Of Parabolic Variational Inequalities}

\author{Yahya Alnashri}
\address[Yahya Alnashri]{Department of Mathematics, Al-Qunfudah University College, Umm Al-Qura University, Saudi Arabia}
\email{yanashri@uqu.edu.sa}


\date{\today}

\begin{document}
\newcommand{\subscript}[2]{$#1 _ #2$}

\begin{abstract}
This work aims to provide a comprehensive and unified numerical analysis for non linear system of parabolic variational inequalities (PVIs) subject to Dirichlet boundary condition. This analysis enables us to establish an existence of the exact solution to the considered model and to prove the convergence for the approximate solution and its approximate gradient. Our results are applicable for several conforming and non conforming numerical schemes.
\end{abstract}

\maketitle


\section{Introduction}
Non linear parabolic variational inequalities and PDEs are useful tools to model the coupled biochemical interactions of microbial cells, which are crucial to numerous applications, especially in the medical field and food production \cite{biofilm-app-1,biofilm-app-2}. We consider here a non linear parabolic system consisting of PDEs and variational inequalities,
\begin{subequations}\label{pvi-obs}
\begin{align}
(\partial_t \bar A-\div({\bf D_A}\nabla\bar A)-F(\bar A,\bar B))(\bar A-\chi)=0 &\mbox{\quad in $\Omega\times(0,T)$,} \label{pvi-obs1}\\
\partial_t \bar A-\div({\bf D_A}\nabla\bar A)\leq F(\bar A,\bar B) &\mbox{\quad in $\Omega\times(0,T)$,} \label{pvi-obs2}\\
\bar A \leq \chi &\mbox{\quad in $\Omega\times(0,T)$,} \label{pvi-obs3}
\\
\partial_t \bar B-\div({\bf D_B}\nabla\bar B)=G(\bar A,\bar B) &\mbox{\quad in $\Omega\times(0,T)$,} \label{pvi-obs4}\\
(\bar A,\bar B) = (0,0) &\mbox{\quad on $(\partial\O\times(0,T))^2$,} \label{pvi-obs5}\\
(\bar A(\x,0), \bar B(\x,0))=(A_{\rm ini},B_{\rm ini})&\mbox{\quad in $(\Omega\times\{0\})^2$}\label{pvi-obs6}.
\end{align}
\end{subequations}

Numerical approximation in parabolic systems of inequalities and generalisation of inequalities have received considerable attention in the research literature. \cite{biofilm-28} obtains the error estimate of second order in $L^\infty(L^2)$ for linear approximation, respect to space and time with a strong regularity on the solutions, such as $\dr_t\bar B\in L^2(0,T,L^2(\O))$. \cite{biofilm-14} analyses Inequality \eqref{pvi-obs1} in which $F=0$ and ${\bf D_A}$ is constant. The work in \cite{biofilm-4} considers a model without the barrier and provides $\mathcal O(h)$ order of convergence in $L^\infty(L^2)$--norm. An $L^2$--error estimate is provided in different studies, such as in \cite{biofilm-27-9} by using finite difference in time. \cite{biofilm-27} deals with parabolic variational inequalities with non linear source term and derives the convergence rate of the finite element method in space respect to $L^\infty$--norm. It also shows that the general finite difference gives $\mathcal O(h)$ in $L^\infty(L^2)$--norm under strong Hypothesis on data.

However, there is a lack in studies investigating the full convergence analysis of numerical schemes for the model \eqref{pvi-obs} since the coupled nonlinearity of the system and the constraint (the inequality) in the model comprise the primary theoretical challenge. It appears that considerable research is still required, beginning with convergence analysis and testing other varieties of scheme outside conforming methods. Rather than undertaking individual research for every numerical scheme, this work utilises a gradient discretisation method (GDM) to afford a unified and full convergence analysis of numerical methods for \eqref{pvi-obs} under the natural Hypothesis on data. The GDM is a generic framework to unify the numerical analysis for diffusion partial differential equations and their corresponding problems. Due to the variety of choice of the discrete elements in the GDM, a series of conforming and non conforming numerical schemes can be included in the GDM, se \cite{31,32,33,34,35,36,37,38} for more details.

\par The outline of this paper is as follows. Section \ref{sec-contin} is devoted to write the model \eqref{pvi-obs} in an equivalent weak sense. Section \ref{sec-disc} defines the discrete space and functions followed by the gradient scheme to our model in weak sense. Section \ref{sec-main} provides the convergence results, Theorem \ref{theorem-conver-rm}, which is proved by following the compactness technique under classical Hypothesis on continuous model data. Finally, as an example, we present in Section \ref{sec-ex} the non conforming $\mathbb P1$ finite element scheme, that has been not yet applied to the non linear model \eqref{pvi-obs}.

\section{Continuous Setting}\label{sec-contin}

\begin{hypothesis}\label{assm-obs}
we assume the following:
\begin{enumerate}
\item $\O \subset \RR^d\; (d\geq 1)$ is a bounded connected open set, and $T>0$,
\item ${\bf D_A}, {\bf D_B}: \O \to \mathbb M_d(\RR)$ are measurable functions (where $\mathbb M_d(\RR)$ consists of $d \times d$ matrices) and there exists $ d_1, \;d_2 >0$ such that for a.e. $\x \in \O$, ${\bf D_A}(\x)$ and ${\bf D_B}(\x)$ are symmetric with eigenvalues in $[d_1, d_2]$,
\item the constraint function $\chi$ is in $H^1(\O)\cap C(\overline\O)$, such that $\chi\geq 0$ on the domain boundary $\dr\O$,  
\item $F$ and $G$ are smooth and Lipschitz functions on $\RR^2$ with Lipschitz constants $M_1$ and $M_2$, respectively, in which $M=\max(M_1,M_2)$,
\item $A_{\rm ini} \in W^{2,\infty}(\O)\cap \cK$, where $\cK:=\{ \varphi \in H_0^1(\O)\; : \; \varphi \leq \chi(t)\; \mbox{ in } \O \}$ and $B_{\rm ini} \in W^{2,\infty}(\O)$.
\end{enumerate}
\end{hypothesis}

With the above Hypothesis, we consider the time dependent closed convex set
\[
\mathbb K:=\{ \varphi\in L^2(0,T;H_0^1(\O)):\; \varphi(t)\in \cK \mbox{ for a.e. } t\in (0,T)\}.
\]
It is clear to see that the time dependent $\mathbb K$ contains at least the constant in time function $t \mapsto \chi^-:=\min(0,\chi)$.

\begin{definition}[Weak formulation]  
Under Hypothesis \ref{assm-obs}, we say that $(\bar A,\bar B)$ is a weak solution of  \eqref{pvi-obs1}--\eqref{pvi-obs6} if the following properties and relations hold:
\begin{enumerate}
\item $\bar A\in \mathbb{K}\cap C^0([0,T];L^2(\O)),\; \bar A(\cdot,0)=A_{\rm ini},\; \dr_t \bar A\in L^2(0,T;L^2(\O))$,
\item $\bar B\in C^0([0,T];L^2(\O)),\; \bar B(\cdot,0)=B_{\rm ini},\; \dr_t \bar B\in L^2(0,T;H^{-1}(\O))$, 
\item for all $\varphi \in \mathbb{K}$, and for all $\psi \in L^2(0,T;H_0^1(\O))$,
\begin{subequations}\label{pvi-obs-weak}
\begin{equation}\label{pvi-obs-w1}
\begin{aligned}
\dsp\int_0^T\dsp\int_\O \partial_t \bar A(\x,t) (A(\x,t)&-\varphi(\x,t)) \ud \x \ud t
+\dsp\int_0^T\int_\O {\bf D_A}\nabla\bar A \cdot \nabla(\bar A-\varphi)(\x,t)\ud \x \ud t\\
&\leq\dsp\int_0^T \int_\O F(\bar A,\bar B)(\bar A(\x,t)-\varphi(\x,t))\ud \x \ud t,\mbox{ and }
\end{aligned}
\end{equation}
\begin{equation}\label{pvi-obs-w2}
\left.
\begin{aligned}
&\dsp\int_0^T \langle \partial_t\bar B(\x,t),\psi(\x,t) \rangle_{H^{-1},H^1}\ud t
+\dsp\int_0^T\int_\O {\bf D_B}(\x)\nabla \bar B(\x,t) \cdot \nabla \psi(\x,t)\ud \x \ud t\\
{}&\quad\quad\quad\quad\quad= \dsp\int_0^T\dsp\int_\O G(\bar A(\x,t),\bar B(\x,t))\psi(\x,t) \ud \x \ud t,
\end{aligned}
\right.
\end{equation}
\end{subequations}
\end{enumerate}
\end{definition}
where $\langle \cdot,\cdot \rangle_{H^{-1},H^1}$ is the duality product between $H^{-1}(\O)$ and $H^1(\O)$.


\section{Discrete Setting}\label{sec-disc}
We begin with defining the discrete space and operators. These discrete elements are slightly different from the ones defined in \cite{38,37}, in particular, $\chi_\disc$, $I_\disc$, and $J_\disc$ are introduced to deal with the non constant barrier $\chi$ and the initial solutions $A_{\ini}$ and $B_{\ini}$.  
 
\begin{definition}[{\bf GD for time dependent for obstacle problem}]\label{def-gd-pvi-obs}
Let $\O$ be an open domain of $\RR^d$ ($d\geq 1$) and $T>0$. A gradient discretisation $\disc$ is defined by $\disc=(X_{\disc,0},\Pi_\disc,\nabla_\disc,\chi_\disc, I_\disc, J_\disc, (t^{(n)})_{n=0,...,N})$, where:
\begin{enumerate}
\item The discrete set $X_{\disc,0}$ is a finite-dimensional vector space on $\RR$, taking into account the homogenous Dirichlet boundary condition \eqref{pvi-obs5}.
\item The linear operator $\Pi_\disc : X_{\disc,0} \to L^2(\O)$ is the reconstruction of the approximate function.
\item The linear operator $\nabla_\disc : X_{\disc,0} \to L^2(\O)^d$ is the reconstruction of the gradient of the function, and must be chosen so that $\| \nabla_\disc\cdot \|_{L^2(\O)^d}$ is a norm on $X_{\disc,0}$,
\item $\chi_\disc\in L^2(\O)$ is an approximation of the barrier $\chi$, 
\item $I_\disc: W^{2,\infty}(\O)\cap\cK \to \cK_\disc:=\{ \varphi\in X_{\disc,0} :\; \Pi_\disc \varphi\leq \chi_\disc, \mbox{ in } \O \}$ is a linear and continuous interpolation operator for the initial solution $A_{\ini}$,
\item $J_\disc: W^{2,\infty}(\O) \to X_{\disc,0}$ is a linear and continuous interpolation operator for the solution $B_{\ini}$, 
\item $t^{(0)}=0<t^{(1)}<....<t^{(N)}=T$.
\end{enumerate}
\end{definition}

\begin{remark}
For a general obstacle $\chi$, most of numerical methods fail to approximate the solution $\bar A$ by elements inside the set $\cK_\disc$. For an example, in the $\mathbb P 1$ finite element method, we consider only the values of $\bar A$ at the vertices of the mesh, which only guarantee that these values satisfy the barrier condition \eqref{pvi-obs3} only at these vertices, not necessarily at any point in $\O$. We define here the set $\cK_\disc$ based on the approximate barrier $\psi_\disc$ to be able to construct an interpolant that belongs to $\cK_\disc$. However, there is no need to use an approximate barrier, if the barrier $\chi$ is assumed to be constant.
\end{remark}

For any $\varphi=(\varphi^{(n)})_{n=0,...,N} \in X_{\disc,0}^{N+1}$, we define the notations as space time functions on as follows: the reconstructed function $\Pi_\disc \varphi: \O\times[0,T]\to \RR$ and the reconstructed gradient $\nabla_\disc \varphi: \O\times[0,T]\to \RR^d$ are given by:
\begin{equation*}
\begin{split}
&\Pi_\disc \varphi(\cdot,0)=\Pi_\disc \varphi^{(0)} \mbox{ and } \forall n=0,...,N-1,\; \forall t\in(t^{(n)},t^{(n+1)}], \forall \x\in \O,\\
&\Pi_\disc \varphi(\x,t)=\Pi_\disc \varphi^{(n+1)}(\x) \mbox{ and }\nabla_\disc \varphi(\x,t)=\nabla_\disc \varphi^{(n+1)}(\x).
\end{split}
\end{equation*}
Setting $\delta t^{(n+\frac{1}{2})}=t^{(n+1)}-t^{(n)}$, for $n=0,...,N-1$, and $\delta t_\disc=\max_{n=0,...,N-1}\delta t^{(n+\frac{1}{2})}$, the discrete derivative $\delta_\disc \varphi \in L^\infty(0,T;L^2(\O))$ of $\varphi\in X_{\disc,\Gamma_2}^{N+1}$ is defined by  
\begin{equation*}
\delta_\disc \varphi(t)=\delta_\disc^{(n+\frac{1}{2})}\varphi:=\frac{\Pi_\disc \varphi^{(n+1)}-\Pi_\disc \varphi^{(n)}}{\delta t^{(n+\frac{1}{2})}}, \mbox{ $\forall n=0,...,N-1$ and $t\in (t^{(n)},t^{(n+1)}]$}.
\end{equation*}

In order to construct a good approximate scheme, we require four properties; coercivity, consistency, limit--conformity, and the compactness. The first three ones respectively connect to the Poincar\'e inequality, the interpolation error and the Stoke formula. The compactness property enables us to deal with non linearity caused by the reaction terms $F$ and $G$.

\begin{definition}[Coercivity]\label{def:corv-rm}
If $\disc$ is a gradient discretisation, set
\[
C_\disc = \dsp\max_{\varphi \in X_{\disc,0}\setminus\{0\}}\frac{\|\Pi_\disc \varphi\|_{L^2(\O)}}{\|\nabla_\disc \varphi\|_{L^2(\O)^d}}.
\]
A sequence $(\disc_m)_{m \in \NN}$ of gradient discretisation is \emph{coercive} if $(C_{\disc_m})_{m\in\NN}$ remains bounded.
\end{definition}

\begin{definition}[Consistency]\label{def:cons-rm}
If $\disc$ is a gradient discretisation, let $S_\disc:\cK \to [0,\infty)$ and $\widetilde{S}_\disc:H_0^1(\O) \to [0,\infty)$ be defined by
\begin{equation}\label{consist-1}
\forall w \in \cK, \; S_\disc(w)=\min_{\varphi\in \cK_\disc} \Big(\| \Pi_\disc \varphi - w \|_{L^2(\O)} 
+ \| \nabla_\disc \varphi - \nabla w \|_{L^2(\O)^{d}}\Big),
\end{equation}
\begin{equation}\label{consist-1}
\forall w \in H_0^1(\O), \; \widetilde S_\disc(w)=\min_{\psi\in X_{\disc,0}} \Big(\| \Pi_\disc \psi - w \|_{L^2(\O)} 
+ \| \nabla_\disc w - \nabla \psi \|_{L^2(\O)^{d}}\Big).
\end{equation}
A sequence $(\mathcal{D}_{m})_{m \in \mathbb{N}}$ of gradient discretisation is \emph{consistent} if, as $m \to \infty$, 
\begin{itemize}
\item for all $w \in \cK$, $S_{\disc_m}(w) \to 0$,
\item for all $w \in H_0^1(\O)$, $\widetilde S_{\disc_m}(w) \to 0$, 
\item for all $w \in W^{2,\infty}(\O)\cap\cK$, $\Pi_{\disc_m}I_{\disc_m}w \to w$ strongly in $L^2(\O)$,
\item for all $w \in W^{2,\infty}(\O)$, $\Pi_{\disc_m}J_{\disc_m}w \to w$ strongly in $L^2(\O)$,
\item $(\| \nabla_{\disc_m}I_{\disc_m} A_{\rm ini}\|_{L^2(\O)^d})_{m\in \NN}$ is bounded,
\item $\delta t_{\disc_m} \to 0$.
\end{itemize}
\end{definition}

\begin{definition}[Limit--conformity]\label{def:lconf-rm}
If $\disc$ is a gradient discretisation, let $W_{\mathcal{D}} : H_{\rm div}(\O):=\{\bpsi \in L^2(\O)^d\;:\; {\rm div}\bpsi \in L^2(\O)\} \to [0, +\infty)$ be defined by
\begin{equation}\label{long-rm}
W_{\mathcal{D}}(\bpsi)
 = \sup_{\varphi\in X_{\disc,0}\setminus \{0\}}\frac{\Big|\dsp\int_{\O}(\nabla_\disc \varphi\cdot \bpsi + \Pi_\disc \varphi \div (\bpsi)) \ud \x \Big|}{\| \nabla_\disc \varphi\|_{L^2(\O)^d} },
\end{equation}
 
A sequence $(\disc_m)_{m\in \NN}$ of gradient discretisation is \emph{limit-conforming} if for all $\bpsi \in H_{\rm div}(\O)$, $W_{\disc_m}(\bpsi) \to 0$, as $m \to \infty$.
\end{definition}

\begin{definition}[Compactness]\label{def:compact}
A sequence of gradient discretisation $(\disc_m)_{m\in\NN}$ is \emph{compact} if for any sequence $(\varphi_m )_{m\in\NN} \in X_{\disc_m,0}$, such that $(\| \nabla_{\disc_m}\varphi_m \|_{L^2(\O)})_{m\in \NN}$ is bounded, the sequence $(\Pi_{\disc_m}\varphi_m )_{m\in \NN}$ is relatively compact in $L^2(\O)$.
\end{definition}

\begin{definition}[Gradient scheme]\label{def-gs-obs} Find sequences $A=( (A^{(n)})_{n=0,...,N}, B=(B^{(n)})_{n=0,...,N} ) \subset \cK_\disc \times X_{\disc,0}$, such that $(A^{(0)},B^{(0)})=( I_\disc A_{\rm ini}, J_\disc B_{\rm ini} )\in \cK_\disc \times X_{\disc,0}$, for all $n=0,...,N-1$, for all $\varphi\in \cK_\disc$, and for all $\psi\in X_{\disc,0}$, 
\begin{subequations}\label{gs-pvi}
\begin{equation}\label{gs-pvi-obs1}
\begin{array}{ll}
\dsp\int_\O \delta_\disc^{(n+\frac{1}{2})}A(\x)\, \Pi_\disc(A^{(n+1)}(\x)-\varphi(\x)) \ud \x\\
+\dsp\int_\O {\bf D_A}(\x)\nabla_\disc A^{(n+1)}(\x)\cdot\nabla_\disc(A^{(n+1)}(\x)-\varphi(\x)) \ud \x\\
\leq \dsp\int_\O F(\Pi_\disc A^{(n+1)}, \Pi_\disc B^{(n+1)})\Pi_\disc(A^{(n+1)}(\x)-\varphi(\x)) \ud \x \ud t, \mbox{ and } 
\end{array}
\end{equation}
\begin{equation}\label{gs-pvi-obs2}
\begin{array}{ll}
\dsp\int_\O \delta_\disc^{(n+\frac{1}{2})}B(\x)\, \Pi_\disc \psi(\x) \ud \x
+\dsp\int_\O {\bf D_B}(\x)\nabla_\disc B^{(n+1)}(\x)\cdot\nabla_\disc \psi(\x) \ud \x\\
= \dsp\int_\O G(\Pi_\disc A^{(n+1)}, \Pi_\disc B^{(n+1)})\Pi_\disc\psi(\x) \ud \x \ud t.
\end{array}
\end{equation}
\end{subequations}
\end{definition}


\section{Main Results}\label{sec-main}
Let the time interval $[0,T]$ be divided into $\ell_\kappa$ intervals of length $\kappa$, where $\kappa$ tends to zero as $\ell_\kappa \to \infty$. Let ${\bf 1}_{I_i}$ be the characteristic function of $I_i=[i\kappa,(i+1)\kappa)$, $i=0,...,\ell_\kappa$. We define a set of piecewise-constant in time functions by
\begin{equation}\label{ch5-dense-set}
\mathbb L_\kappa=\left\{ w_\kappa(\x,t)=\sum_{i=1}^{\ell_\kappa}{\bf 1}_{I_i}(t)\varphi_i(\x)\; :\; \varphi \in C_0^2(\overline\O) \mbox{ and $\varphi \leq \chi$ in $\O$, a.e.} \right\}.
\end{equation}

\begin{lemma}\label{ch5-lem-intp}
For $T>0$, let $(\disc)_{m\in\NN}$ be a sequence of gradient discretisation, that is consistent. Let $\bar w_\kappa\in \mathbb L_\kappa$ be a piecewise constant in time function, where $\mathbb L_\kappa$ is the set defined by \eqref{ch5-dense-set}. Then there exists a sequence $(w_m)_{m\in\NN}$ such that $w_m=(w_m^{(n)})_{n=0,...,N_m} \in \cK_{\disc_m}^{N_m+1}$ for all $m\in\NN$, and, as $m\to \infty$,
\begin{subequations}
\begin{align}
&\Pi_{\disc_m}w_m \to \bar w_\kappa \mbox{ strongly in }  L^2(\O\times (0,T)),\label{ch5-pvi-wekinterp1}\\
&\nabla_{\disc_m}w_m \to \nabla\bar w_\kappa \mbox{ strongly in }  L^2(\O\times (0,T))^d.\label{ch5-pvi-wekinterp2}
\end{align}
\end{subequations}
\end{lemma}

\begin{proof}[{\bf Proof}]
Write $\bar w_\kappa(\x,t)=\sum_{i=1}^{\ell_\kappa} {\bf 1}_{I_i}(t)\phi_i(\x)$ such that $\phi_i \in C_0^\infty(\overline \O)\cap \cK$. Let $s\in (0,T)$ and choose $n:=n(s)$ such that $s\in (t^{(n(s))},t^{(n(s)+1)}]$. Let $w_m\in X_{\disc_m,0}$ be defined by $w_m=\sum_{i=1}^{\ell_\kappa}{\bf 1}_{I_i}(t^{(n(s)+1)})P_{\disc_m}\phi_i$, where
\begin{equation}\label{ch5-cons-funct}
P_{\disc_m}(\phi)  ={ \displaystyle \argminB_{\omega \in \cK_{\disc_m}
}\left(\| \Pi_{\disc_m} \omega - \phi \|_{L^2(\Omega)} + \| \nabla_{\disc_m} \omega - \nabla \phi \|_{L^2(\Omega)^{d}}\right)}.
\end{equation}
For $i=1,...,\ell_\kappa$, we define ${\xi}_m^i:(0,T) \to \RR$ by ${\xi}_m^i(s)={\bf 1}_{I_i}(t^{(n(s)+1)})$ for $s\in (0,T)$. Using the relation $ab-cd=(a-c)b+c(b-d)$, we obtain, for all $s\in (0,T)$ and a.e. $\x\in\O$,
\begin{equation*}
\begin{aligned}
(\Pi_{\disc_m}w_m-\bar w_\kappa)(\x,s)&= \sum_{i=1}^{\ell_\kappa} \left({\xi}_m^i(s)-{\bf 1}_{I_i}(s) \right)\Pi_{\disc_m}P_{\disc_m}\phi_i(\x)\\
&+\sum_{i=1}^{\ell_\kappa} {\bf 1}_{I_i}(s) \left( \Pi_{\disc_m}P_{\disc_m}\phi_i -\phi_i \right)(\x).
\end{aligned}
\end{equation*}
An application of the definition of $S_{\disc_m}$ yields
\begin{equation}\label{ch5-eq-interp-lemma-proof}
\begin{aligned}
\| \Pi_{\disc_m}w_m-\bar w_\kappa \|_{L^2(\O\times (0,T))}&\leq \sum_{i=1}^{\ell_\kappa}\| {\xi}_m^i(s)-{\bf 1}_{I_i}(s) \|_{L^2(0,T)} \| \Pi_{\disc_m}P_{\disc_m}\phi_i
\|_{L^2(\O)}\\
&+\sum_{i=1}^{\ell_\kappa}\| {\bf 1}_{I_i}(s) \|_{L^2(0,T)} \| \Pi_{\disc_m}P_{\disc_m}\phi_i -\phi_i \|_{L^2(\O)}\\
&\leq \sum_{i=1}^{\ell_\kappa} \| {\xi}_m^i(s)-{\bf 1}_{I_i}(s) \|_{L^2(0,T)} \left( S_{\disc_m}(\phi_i)+\|\phi_i\|_{L^2(\O)} \right)\\
&+C_1\sum_{i=1}^{\ell_\kappa}S_{\disc_m}(\phi_i),
\end{aligned}
\end{equation}
where $C_1=\sum_{i=1}^{\ell_\kappa} \|{\bf 1}_{I_i} \|_{L^2(0,T)}$. From the consistency, one obtains $S_{\disc_m}(\phi_i)\to 0$ as $m\to \infty$, for any $i=0,...,\ell_\kappa$, which implies that the second term on the R.H.S vanishes. In the case in which both $s$, $t^{(n(s)+1)} \in I_i$ or both $s$, $t^{(n(s)+1)} \notin I_i$, the quantity ${\xi}_m^i(s)-{\bf 1}_{I_i}(s)$ equals zero. In the case in which $s\in I_i$ and $t^{(n(s)+1)} \notin I_i$ or $s\notin I_i$ and $t^{(n(s)+1)} \in I_i$, one can deduce (writing $I_i=[a_i,b_i]$ and because $s$ is chosen such that $|s-t^{(n(s)+1)}|\leq \delta t_{\disc_m}$
\begin{align*}
\| {\xi}_m^i(s)-{\bf 1}_{I_i}(s) \|_{L^2(0,T)}^p &\leq {\rm measure}([a_i-\delta t_{\disc_m},a_i+\delta t_{\disc_m}] \cup [b_i-\delta t_{\disc_m}, b_i+\delta t_{\disc_m}])\\
&\leq 4 \delta t_{\disc_m}.
\end{align*}
This shows that the first term on the R.H.S of \eqref{ch5-eq-interp-lemma-proof} tends to zero when $m \to \infty$. Hence, \eqref{ch5-pvi-wekinterp1} is concluded. The proof of \eqref{ch5-pvi-wekinterp2} is obtained by the same reasoning, replacing $\bar w_\kappa$ by $\nabla\bar w_\kappa$ and $\Pi_{\disc_m}w_m$ by $\nabla_{\disc_m}w_m$.     
\end{proof}


\begin{lemma}[{\bf Energy estimates}]
\label{ch5-lem-derv}
Let Hypothesis \ref{assm-obs} hold. If $\disc$ is a gradient discretisation, such that $\delta_\disc < \frac{1}{2M}$, $\cK_\disc$ is a non empty set, and $(A,B)\in\cK_\disc \times X_{\disc,0}$ is a solution of the approximate scheme \eqref{gs-pvi}, then there exists a constant $C_2\geq 0$ only depending on $\O$, $d_1$, $T$, $M$, $C_0:=\max(F({\bf 0}), G({\bf 0}))$, $\| \Pi_{\disc}I_\disc A_{\ini}\|_{L^2(\O)}$, $\| \nabla_{\disc}I_\disc\nabla A_{\ini} \|_{L^2(\O)^d}$, and $\| \Pi_{\disc}J_\disc B_{\ini} \|_{L^2(\O)}$, such that
\begin{equation}\label{est-new}
\begin{aligned}
\|\delta_\disc A \|_{L^2(\O\times(0,T))} 
&+\|\nabla_\disc A \|_{L^\infty(0,T;L^2(\O)^d)}\\
&+\|\Pi_\disc B\|_{L^\infty(0,T;L^2(\O))}+ \| \nabla_\disc B \|_{L^2(\O\times(0,T))^d}\\
&\leq C_2.
\end{aligned}
\end{equation}
\end{lemma}

\begin{proof}[{\bf Proof}]
We start by taking $\varphi:=A^{(n)}$ (it belongs to $\cK_\disc$) and the function $\psi:=\delta t^{ (n+\frac{1}{2}) }B^{(n+1)}$ in \eqref{gs-pvi} to get
\begin{equation}\label{ch5-nw-eq1}
\begin{aligned}
\delta t^{(n+\frac{1}{2})}\dsp\int_\O |\delta_\disc^{(n+\frac{1}{2})}A|^2 \ud \x
&+\dsp\int_\O {\bf D_A}\nabla_\disc A^{(n+1)}\cdot \nabla_\disc(A^{(n+1)}-A^{(n)}) \ud \x\\
&\quad\leq \delta t^{(n+\frac{1}{2})}\dsp\int_\O F(\Pi_\disc A^{(n+1)},\Pi_\disc B^{(n+1)})\delta_\disc^{(n+\frac{1}{2})}A  \ud \x,
\end{aligned}
\end{equation}
and
\begin{equation}\label{ch5-nw-eq1-1}
\begin{aligned}
\dsp\int_\O \Big(\Pi_\disc B^{(n+1)}(\x)&-\Pi_\disc B^{(n)}(\x)\Big) \Pi_\disc B^{(n+1)}(\x) \ud \x
+\dsp\dsp\int_{t^{(n)}}^{t^{(n+1)}}\int_\O {\bf D_B}|\nabla_\disc B^{(n+1)}(\x)|^2 \ud \x \ud t\\
&\quad= \delta t^{(n+\frac{1}{2})}\dsp\int_\O G(\Pi_\disc  A^{(n+1)}(\x), B^{(n+1)}(\x)) \Pi_\disc B^{(n+1)}(\x)  \ud \x.
\end{aligned}
\end{equation}
Applying the fact that $(r-s)\cdot r \geq \frac{1}{2}|r|^2-\frac{1}{2}|s|^2$ to the second term on the L.H.S of \eqref{ch5-nw-eq1} and to the first term on the L.H.S of \eqref{ch5-nw-eq1-1}, it follows
\begin{equation*}
\begin{aligned}
\delta t^{(n+\frac{1}{2})}\dsp\int_\O |\delta_\disc^{(n+\frac{1}{2})}A|^2 \ud \x
&+\frac{d_1}{2}\dsp\int_\O \left( |\nabla_\disc A^{(n+1)}|^2 
-|\nabla_\disc A^{(n)}|^2 \right)\ud \x\\
&\quad\leq \delta t^{(n+\frac{1}{2})}\dsp\int_\O F(\Pi_\disc A^{(n+1)},\Pi_\disc B^{(n+1)})\delta_\disc^{(n+\frac{1}{2})}A  \ud \x,
\end{aligned}
\end{equation*}
and 
\begin{equation*}
\begin{aligned}
\frac{1}{2}\dsp\int_\O \Big[ |\Pi_\disc B^{(n+1)}(\x)|^2&-|\Pi_\disc B^{(n)}(\x)|^2 \Big] \ud \x
+d_1\dsp\int_{t^{(n)}}^{t^{(n+1)}}\int_\O |\nabla_\disc B^{(n+1)}(\x)|^2 \ud \x \ud t\\
&\quad\leq \delta t^{(n+\frac{1}{2})}\int_\O G(\Pi_\disc  U^{(n+1)}(\x), B^{(n+1)}(\x)) \Pi_\disc B^{(n+1)}(\x)  \ud \x.
\end{aligned}
\end{equation*}
Summing the above inequalities over $n\in[0,m-1]$, where $m=0,...,N$ gives
\begin{equation}\label{ch5-new-eq2}
\begin{aligned}
\|\delta_\disc A\|_{L^2(\O\times(0,t^{(m)}))}^2
&+\frac{d_1}{2}\Big(\| \nabla_\disc A^{(m)} \|_{L^2(\O)^d}^2-\| \nabla_\disc A^{(0)} \|_{L^2(\O)^d}^2\Big)\\
&\leq \dsp\sum_{n=0}^{m-1}\delta t^{(n+\frac{1}{2})}\int_\O F(\Pi_\disc A^{(n+1)},\Pi_\disc B^{(n+1)}) \delta_\disc^{(n+\frac{1}{2})} A  \ud \x,
\end{aligned}
\end{equation}
and
\begin{equation}\label{eq-est10}
\begin{aligned}
\frac{1}{2}\Big(\| \Pi_\disc B^{(m)} \|_{L^2(\O)}^2&-\| \Pi_\disc B^{(0)} \|_{L^2(\O)}^2\Big)
+d_2\dsp\sum_{n=0}^{m-1}\delta t^{(n+\frac{1}{2})} \| \nabla_\disc B^{(n)}\|_{L^2(\O)^d}^2\\
&\quad\leq \dsp\sum_{n=0}^{m-1}\delta t^{(n+\frac{1}{2})}\int_\O G(\Pi_\disc  A^{(n+1)}(\x), B^{(n+1)}(\x)) \Pi_\disc B^{(n+1)}(\x)  \ud \x.
\end{aligned}
\end{equation}
This with the Cauchy--Schwarz inequality imply that
\begin{equation*}
\begin{aligned}
&\|\delta_\disc A\|_{L^2(\O\times(0,t^{(m)}))}^2
+\frac{d_1}{2}\Big(\| \nabla_\disc A^{(m)} \|_{L^2(\O)^d}^2-\| \nabla_\disc A^{(0)} \|_{L^2(\O)^d}^2\Big)\\
&\leq\dsp\sum_{n=0}^{m}\delta t^{(n+\frac{1}{2})}\| F(\Pi_\disc  A^{(n+1)}, \Pi_\disc B^{(n+1)})\|_{L^2(\O\times (0,T))}\|\delta_\disc^{(n+\frac{1}{2})} A \|_{L^2(\O\times(0,t))},
\end{aligned}
\end{equation*}
and  
\begin{equation*}
\begin{aligned}
&\frac{1}{2}\Big(\| \Pi_\disc B^{(m)} \|_{L^2(\O)}^2-\| \Pi_\disc B^{(0)} \|_{L^2(\O)}^2\Big)
+d_2\dsp\sum_{n=0}^{m-1}\delta t^{(n+\frac{1}{2})} \| \nabla_\disc B^{(n)}\|_{L^2(\O)^d}\\
&\quad\leq \dsp\sum_{n=0}^{m}\delta t^{(n+\frac{1}{2})}\| G(\Pi_\disc  A^{(n+1)}, \Pi B^{(n+1)})\|_{L^2(\O\times (0,T))} \; \| \Pi_\disc B^{(n+1)}\|_{L^2(\O\times (0,T))}.
\end{aligned}
\end{equation*}
Use the Lipschitz condition to arrive at
\begin{equation*}
\begin{aligned}
\|\delta_\disc A\|_{L^2(\O\times(0,t^{(m)}))}^2
&+\frac{d_1}{2}\Big(\| \nabla_\disc A^{(m)} \|_{L^2(\O)^d}^2-\| \nabla_\disc A^{(0)} \|_{L^2(\O)^d}^2\Big)\\
&\leq \dsp\sum_{n=0}^{m}\delta t^{(n+\frac{1}{2})} \|\delta_\disc^{(n+\frac{1}{2})} A \|_{L^2(\O\times(0,t))}\Big(M \| \Pi_\disc A^{(n+1)} \|_{L^2(\O)}\\
&+M \| \Pi_\disc B^{(n+1)} \|_{L^2(\O)}+ C_0\Big),
\end{aligned}
\end{equation*}
and
\begin{equation*}
\begin{aligned}
\frac{1}{2}\Big(\| \Pi_\disc B^{(m)} \|_{L^2(\O)}^2&-\| \Pi_\disc B^{(0)} \|_{L^2(\O)}^2\Big)
+d_2\dsp\sum_{n=0}^{m-1}\delta t^{(n+\frac{1}{2})} \| \nabla_\disc B^{(n)}\|_{L^2(\O)^d}\\ 
&\leq \dsp\sum_{n=0}^{m}\delta t^{(n+\frac{1}{2})}\Big(M \| \Pi_\disc B^{(n+1)} \|_{L^2(\O)}^2
\\
&+M \| \Pi_\disc B^{(n+1)} \|_{L^2(\O)} \| \Pi_\disc B^{(n+1)} \|_{L^2(\O)}\\
&+C_0 \| \Pi_\disc B^{(n+1)} \|_{L^2(\O)} \Big).
\end{aligned}
\end{equation*}
This with the Young's inequality give, where $1-\sum_{i=1}^3\varepsilon_i>0$
\begin{equation}\label{eq-new-lemma}
\begin{aligned}
\|\delta_\disc A\|_{L^2(\O\times(0,t^{(m)}))}^2
&+\frac{d_1}{2}\Big(\| \nabla_\disc A^{(m)} \|_{L^2(\O)^d}^2-\| \nabla_\disc A^{(0)} \|_{L^2(\O)^d}^2\Big)\\
&\leq \dsp\sum_{i=1}^3\frac{\varepsilon_i}{2}\|\delta_\disc A^{(n+\frac{1}{2})} \|_{L^2(\O\times(0,t))}^2\\
&+\dsp\sum_{n=0}^{m}\delta t^{(n+\frac{1}{2})}\Big(\frac{M}{2\varepsilon_1}\| \Pi_\disc A^{(n+1)} \|_{L^2(\O)}^2\\
&+\frac{M}{2\varepsilon_2}\| \Pi_\disc B^{(n+1)} \|_{L^2(\O)}^2+ \frac{1}{2\varepsilon_3}C_0\Big),
\end{aligned}
\end{equation}
and
\begin{equation}\label{eq-new-lemma-22}
\begin{aligned}
\frac{1}{2}\Big(\| \Pi_\disc B^{(m)} \|_{L^2(\O)^d}^2&-\| \Pi_\disc B^{(0)} \|_{L^2(\O)^d}^2\Big)
+d_2\dsp\sum_{n=0}^{m-1}\delta t^{(n+\frac{1}{2})} \| \nabla_\disc B^{(n)}\|_{L^2(\O)^d}^2\\
&\quad\leq \dsp\sum_{n=0}^{m}\delta t^{(n+\frac{1}{2})}
\Big( M(1+\frac{1}{2\varepsilon_4})\| \Pi_\disc B^{(n+1)} \|_{L^2(\O)}^2 \\
&+ \frac{\varepsilon_4}{2}\| \Pi_\disc A^{(n+1)} \|_{L^2(\O)}^2+ \frac{\varepsilon_5}{2}C_0^2 \Big).
\end{aligned}
\end{equation}
Thanks to the Gronwall inequality \cite[Lemma 5.1]{Gronwall}, Inequality \eqref{eq-new-lemma-22} can be rewritten as
\begin{equation*}
\begin{aligned}
\frac{1}{2}\| \Pi_\disc B^{(m)} \|_{L^2(\O)^d}^2&+d_2\dsp\sum_{n=0}^{m-1}\delta t^{(n+\frac{1}{2})} \| \nabla_\disc B^{(n)}\|_{L^2(\O)^d}^2
\\
&\quad\leq {\bf e}^{C_3}\Big(\frac{T\varepsilon_5}{2}C_0^2+\frac{\varepsilon_4}{2}\dsp\sum_{n=0}^{m-1}\delta t^{(n+\frac{1}{2})}\| \Pi_\disc A^{(n+1)} \|_{L^2(\O\times (0,T))}^2\\
&+\frac{1}{2}\| \Pi_\disc B^{(0)} \|_{L^2(\O)}^2\Big),
\end{aligned}
\end{equation*}
where $C_3$ depends on $T$, $M$, and $\varepsilon_4$. Combining this inequality together with \eqref{eq-new-lemma} yields
\begin{equation*}
\begin{aligned}
\Big(2&-\dsp\sum_{i=1}^3\frac{\varepsilon_i}{2}\Big)\|\delta_\disc A \|_{L^2(\O\times(0,t^{(m)}))}^2
+\Big(\frac{d_1}{2}+\frac{M}{2\varepsilon_1}-\frac{\varepsilon_4}{2}{\bf e}^{C_3}\Big)\|\nabla_\disc A^{(m)}\|_{L^2(\O)^d}^2\\
&+\frac{1}{2} \|\Pi_\disc B^{(m)}(\x)\|_{L^2(\O)}^2
+\Big(d_2-\frac{M}{2\varepsilon_2} \|\nabla_\disc B(\x,t)\|_{L^2(\O\times(0,t^{(m)}))^d}^2\\
&\leq\Big(\frac{T}{2\varepsilon_3}+\frac{T\varepsilon_5}{2}{\bf e}^{C_3}\Big)C_0^2
+\frac{d_1}{2}\| \nabla_\disc A^{(0)} \|_{L^2(\O)^d}^2
+\frac{{\bf e}^{C_3}}{2}\| \Pi_\disc B^{(0)} \|_{L^2(\O)}^2.
\end{aligned}
\end{equation*}
Take the supremum over $m\in[0,N]$ and use the real inequality $\sup_n(r_n+s_n) \leq \sup_n(r_n)+\sup{_n}(s_n)$ to obtain the desired estimates.
\end{proof}

In the following definition, we introduce a dual norm \cite{30}, which is defined on the space $\Pi_\disc(X_{\disc,0}) \subset L^2(\O)$, to ensure the required compactness results.

\begin{definition}\label{def-dual}
If $\disc$ be a gradient discretisation, then the dual norm $\| \cdot \|_{\star,\disc}$ on $\Pi_\disc(X_{\disc,0})$ is given by
\begin{equation}\label{eq-dual}
\forall U\in \Pi_\disc(X_{\disc,0}),\;
\| U \|_{\star,\disc}=\dsp\sup\Big\{ \dsp\int_\O U(\x)\Pi_\disc \psi(\x)\ud \x \; :\; \psi\in X_{\disc,0}, \| \nabla_\disc \psi \|_{L^2(\O)^d} =1
\Big\}. 
\end{equation}
\end{definition}

\begin{lemma}
\label{lemma-est-dual}
under Hypothesis \ref{assm-obs}, let $\disc$ be a gradient discretisation, which is coercive. If $B \in X_{\disc,0}$ satisfies \eqref{gs-pvi-obs2}, then there exists a constant $C_4$ depending only on $C_1$, $M$, $\O$, $T$ and $\| \Pi_\disc B^{(0)} \|_{L^2(\O)}$, such that
\begin{equation}\label{eq-est-dual}  
\dsp\int_0^T \| \delta_\disc B(t) \|_{\star,\disc}^2 \ud t \leq C_4.
\end{equation}
\end{lemma}

\begin{proof}
Putting $\psi =\phi$ in \eqref{gs-pvi-obs2} together with the Cauchy--Schwarz inequality the coercivity property imply 
\[
\begin{aligned}
\dsp\int_\O \delta_\disc^{(n+\frac{1}{2})} &B(\x) \Pi_\disc \phi(\x) \ud \x\\
&\quad\leq d_2 \| \nabla_\disc B^{(n+1)} \|_{L^2(\O\times(0,T))^d} \| \nabla_\disc \phi \|_{L^2(\O\times(0,T))^d}\\
&\qquad+ \left( M \| \Pi_\disc B^{(n+1)} \|_{L^2(\O\times(0,T))}
+ M \| \Pi_\disc A^{(n+1)} \|_{L^2(\O\times(0,T))} + C_0 \right) \| \Pi_\disc \phi\|_{L^2(\O)}\\
&\quad\leq  \| \nabla_\disc \phi \|_{L^2(\O)^d} \Big[ d_2 \| \nabla_\disc B^{(n+1)} \|_{L^2(\O\times(0,T))^d}\\
&\qquad+C_\disc \Big( M \| \Pi_\disc B^{(n+1)} \|_{L^2(\O\times(0,T))} + M \| \Pi_\disc A^{(n+1)} \|_{L^2(\O\times(0,T))}
+ C_0 \Big) \Big].
\end{aligned}
\]
Taking the supremum over $\phi\in X_{\disc,0}$ with $\| \nabla_\disc \phi \|_{L^2(\O)^d}=1$, multiplying by $\delta t^{(n+1)}$, summing over $n\in[0,N-1]$, and using \eqref{est-new} yield the desired estimate.
\end{proof}


\begin{theorem}\label{theorem-conver-rm}
With Hypothesis \eqref{assm-obs}, let $(\disc_m)_{m\in\NN}$ be a sequence of gradient discretisation, that is coercive, limit-conforming, consistent and compact, and such that $\cK_{\disc_m}$ is a non empty set for any $m\in\NN$. For $m \in \NN$, let $(A_m,B_m) \in \cK_{\disc_m}^{N_m+1} \times X_{\disc_m,0}^{N_m+1}$ be solutions to the scheme \eqref{gs-pvi} with $\disc=\disc_m$. Then there exists a solution $(\bar A,\bar B)$ for the discrete problem \eqref{pvi-obs-weak}, and a subsequence of gradient discretisation indexed by $(\disc_m)_{m\in\NN}$, such that, as $m\to\infty$,
\begin{enumerate}
\item $\Pi_{\disc_m}A_m\to A$ and $\Pi_{\disc_m}B_m\to B$ strongly in $L^\infty(0,T;L^2(\O))$,
\item $\nabla_{\disc_m}A_m\to \nabla A$ and $\nabla_{\disc_m}B_m\to \nabla B$ strongly in $L^2(\O\times (0,T))^d$,
\item $\delta_{\disc_m}A_m$ converges weakly to $\partial_t\bar A$ in $L^2(\O\times(0,T))$.
\end{enumerate}
\end{theorem}

\begin{proof}
The proof is divided into four stages and its idea is inspired by \cite{YH-2021-P1}.

\medskip
\noindent\textbf{Step 1: Existence of approximate solutions.}\\
At $(n+1)$, we see that \eqref{gs-pvi-obs1} and \eqref{gs-pvi-obs2} respectively express a system of non linear elliptic variational inequality on $A^{(n+1)}$ and non linear equations on $B^{(n+1)}$. For $w=(w_1,w_2) \in  \cK_\disc \times X_{\disc,0}$, we see that $(A,B)\in \cK_\disc \times X_{\disc,0}$ satisfies
\begin{subequations}\label{gs-pvi-sol}
\begin{equation}\label{gs-pvi-obs1-sol}
a(A^{(n+1)},A^{(n+1)}-\varphi) \leq L(A^{(n+1)}-\varphi), \quad \forall \varphi \in \cK_\disc,
\mbox{ and }
\end{equation}
\begin{equation}\label{gs-pvi-obs2-sol}
\begin{array}{ll}
\dsp\int_\O \Pi_{\disc}\dsp\frac{B^{(n+1)}-B^{(n)}}{\delta t^{(n+\frac{1}{2})}}(x) \Pi_\disc \psi(\x)
+ \dsp\int_\O {\bf D_B}\nabla_\disc B^{(n+1)}(\x) \cdot \nabla_\disc \psi(\x)\ud \x\\
= \dsp\int_\O G(\Pi_\disc w_1, \Pi_{\disc} w_2)\Pi_\disc \psi(\x) \ud \x, \quad \forall \psi \in X_{\disc,0},
\end{array}
\end{equation}
\end{subequations}
where $\alpha:=\frac{1}{\delta t^{(n+\frac{1}{2})}}$ and the bilinear and linear forms are defined by 
\[
a(\phi,z)=\alpha \dsp\int_\O \Pi_\disc \phi  \Pi_\disc z \ud \x +{\bf D_A}\int_\O \nabla_\disc \phi \cdot \nabla_\disc z \ud \x,\; \forall \phi,z \in \cK_\disc \quad \mbox{ and }
\]
\[L(z)=\dsp\int_\O F(\Pi_\disc w_1, \Pi_\disc w_2) \Pi_\disc z \ud \x + \alpha \int_\O \Pi_\disc A^{(n)} \Pi_\disc z \ud \x,\; \forall z \in \cK_\disc.
\]

Stampacchia's theorem implies that there exists  $\bar A\in \cK_\disc$ satisfying the inequality \eqref{gs-pvi-obs1-sol}. The second equation \eqref{gs-pvi-obs2-sol} describes a linear square system. Taking $\varphi=B^{(n+1)}$ in \eqref{gs-pvi-obs2-sol}, using the similar reasonings to the proof of Lemma \ref{ch5-lem-derv}, and setting $G={\bf 0}$ yield $\| \nabla_\disc B^{(n+1)} \|_{L^2(\O)^d}=0$. This yields the matrix corresponding to the linear system is invertible. Consider the continuous  mapping $\mathbb T:\cK_\disc \times X_{\disc,0} \to \cK_\disc \times X_{\disc,0}$, where $\mathbb T(w)=(A,B)$ with $(A,B)$ is the solution to \eqref{gs-pvi-sol}. The existence of a solution $(A^{(n+1)},B^{(n+1)})$ to the non linear system is consequence of Brouwer's fixed point theorem.

\vskip .5pc
\noindent\textbf{Step 2: Strong convergence of $\Pi_{\disc_m}A_m$ and $\Pi_{\disc_m}B_m$ in $L^\infty(0,T;L^2(\O))$ and the weak convergence of $\delta_{\disc_m}A_m$ in $L^2(\O\times(0,T))$.}
\\
Applying Estimate \eqref{est-new} to the sequence of solutions $((A_m)_{m\in\NN},(B_m)_{m\in\NN})$ of the scheme \eqref{gs-pvi} shows that $\| \nabla_{\disc_m}A_m \|_{L^2(\O\times(0,T))^d}$ and $\| \nabla_{\disc_m}B_m \|_{L^2(\O\times(0,T))^d}$ are bounded. Using \cite[Lemma 4.8]{30}, there exists a sequence, still denoted by $(\disc_m^T)_{m\in\NN}$, and $\bar A,\; \bar B\in L^2(0,T;H_0^1(\O))$,
such that, as $m\to \infty$, $\Pi_{\disc_m}A_m$ converges weakly to $\bar A$ in $L^2(\O\times(0,T))$, $\nabla_{\disc_m}A_m$ converges weakly to
$\nabla\bar A$ in $L^2(\O\times(0,T))^d$, $\Pi_{\disc_m}B_m$ converges weakly to $\bar B$ in $L^2(\O\times(0,T))$ and $\nabla_{\disc_m}B_m$ converges weakly to
$\nabla\bar B$ in $L^2(\O\times(0,T))^d$. Since $A_m \in \cK_{\disc_m}$, passing to the limit in $\Pi_{\disc_m}A_m \leq \chi$ in $\O$ shows that $\bar A \leq \chi$ in $\O$. Thanks to \cite[Theorem 4.31]{30}, Estimate \eqref{est-new} shows that $\bar A\in C([0,T];L^2(\O))$, $\Pi_{\disc_m}A_m$ converges strongly to $\bar A$ in $L^\infty(0,T;L^2(\O))$ and $\delta_{\disc_m}A_m$ converges weakly to $\partial_t \bar A$ in $L^2(0,T;L^2(\O))$. 

Let us show the strong convergence of $\Pi_{\disc_m}B_m$ to $\bar B$ in $L^\infty(0,T;L^2(\O))$. Indeed, $\Pi_{\disc_m}B_m$ converges strongly to $\bar B$ in $L^2(\O\times(0,T))$, thanks to Estimate \eqref{eq-est-dual}, consistency, limit--conformity and compactness, and \cite[Theorem 4.14]{30}. We can apply the dominated convergence theorem to show that $G(\Pi_{\disc_m}A_m,\Pi_{\disc_m}B_m) \to G(\bar A,\bar B)$ in $L^2(\O\times(0,T))$, thanks to the hypothesis on $G$ given in Hypothesis \ref{assm-obs}.

Let $t_0 \in [0,T]$ and define the sequence $t_m \in [0,T]$, such that $t_m \to t_0$, as $m \to \infty$. For $k(m) \in [0,N_m-1]$ such that $k_m 
\in (t^{(s(m))}, t^{(s(m)+1)}]$. Following technique used in Lemma \ref{ch5-lem-derv}, one can attain 
\begin{equation}\label{eq-prf-thm1}
\begin{aligned}
\frac{1}{2}\dsp\int_\O &(\Pi_{\disc_m} B(\x,t_m))^2 \ud \x\\
&\leq \frac{1}{2}\dsp\int_\O (\Pi_{\disc_m} J_{\disc_m}B_{\rm ini}(\x))^2 \ud x 
- \dsp\int_0^{t^{(s(m))}}\dsp\int_\O {\bf D_B}(\nabla_{\disc_m}B(\x,t))^2 \ud \x \ud t\\
&\leq \dsp\int_0^{t^{(s(m))}}\dsp\int_\O G(\Pi_{\disc_m}A(\x,t),\Pi_{\disc_m}B(\x,t)) \Pi_{\disc_m}B(\x,t) \ud \x \ud t.
\end{aligned}
\end{equation}

Using the characteristic function, it is obvious to note that, as $m \to \infty$,
\begin{equation*}
\begin{aligned}
&\Pi_{\disc_m}B_m \to \bar B \mbox{ strongly in } L^2(\O \times (0,T)), \mbox{ and,}\\
&{\bf 1}_{[0,t^{(s(m))}]}\nabla\bar B \to {\bf 1}_{[0,t_0]} \nabla\bar B \mbox{ strongly in } L^2(\O \times (0,T))^d,
\end{aligned}
\end{equation*}
These convergence results imply that 
\begin{equation*}
\begin{aligned}
\dsp\int_0^{t_0}\dsp\int_\O &{\bf D_B}(\nabla\bar B(\x,t))^2 \ud \x \ud t\\
&=\dsp\int_0^{t^{(s(m))}}\dsp\int_\O {\bf 1}_{[0,s]}{\bf D_B}(\nabla\bar B(\x,t))^2 \ud \x \ud t\\
&=\dsp\lim_{m\to \infty}\dsp\int_0^T\dsp\int_\O {\bf 1}_{[0,t^{(s(m))}]}{\bf D_B}\nabla\bar B(\x,t) \cdot \nabla_{\disc_m}B_m(\x,t) \ud \x \ud t\\
&\leq \dsp\liminf_{m\to \infty}\Big( \| {\bf 1}_{[0,t^{(s(m))}]}\nabla\bar B \|_{L^2(\O\times(0,T))^d} \cdot \| {\bf 1}_{[0,t^{(s(m))}]}{\bf D_B}\nabla_{\disc_m}B_m \|_{L^2(\O\times(0,T))^d} \Big)\\
&= \| {\bf 1}_{[0,t_0]}\nabla\bar B \|_{L^2(\O\times(0,T))^d} \cdot \dsp\liminf_{m\to \infty} \| {\bf 1}_{[0,t^{(s(m))}]}{\bf D_B}\nabla_{\disc_m}B_m \|_{L^2(\O\times(0,T))^d}. 
\end{aligned}
\end{equation*}
Dividing this inequality by $\| {\bf 1}_{[0,t_0]}\nabla\bar B \|_{L^2(\O\times(0,T))^d}$ gives
\begin{equation}\label{eq-liminf1}
\begin{aligned}
\dsp\int_0^{t_0}\dsp\int_\O {\bf D_B}(\nabla\bar B(\x,t))^2 \ud \x \ud t \leq \dsp\liminf_{m\to\infty}\dsp\int_0^{t^{(s(m))}}\dsp\int_\O {\bf D_B}( \nabla_{\disc_m}B_m(\x,t))^2 \ud \x \ud t.
\end{aligned}
\end{equation}
Together with passing to limit superior in \eqref{eq-prf-thm1}, we can arrive at
\begin{equation}\label{eq-20}
\begin{aligned}
\dsp\limsup_{m\to \infty}\frac{1}{2}\dsp\int_\O (\Pi_{\disc_m}&B_m(\x,t_m))^2 \ud \x\\
&\leq \frac{1}{2}\dsp\int_\O B_{\rm ini}(\x)^2 \ud \x
-\dsp\int_0^{t_0}\dsp\int_\O {\bf D_B} (\nabla\bar B(\x,t))^2 \ud \x \ud t\\
&\quad+\dsp\int_0^{t_0}\dsp\int_\O G(\bar A(\x,t),\bar B(\x,t)) \bar B(\x,t)
 \ud \x \ud t.
\end{aligned}
\end{equation}
Letting $\psi=\bar B{\bf 1}_{[0,t_0]}(t)$ in \eqref{gs-pvi-obs2} and integrating by part to obtain
\begin{equation}\label{eq-multline-1}
\begin{aligned}
\frac{1}{2}\dsp\int_\O (\bar B(\x,t_0))^2 \ud \x &+ \dsp\int_0^{t_0}\dsp\int_\O {\bf D_B}(\nabla \bar B(\x,t))^2 \ud \x \ud t\\
&\quad=\frac{1}{2}\dsp\int_\O B_{\rm ini}(\x)^2 \ud \x
+\dsp\int_0^{t_0}\dsp\int_\O G(\bar A(\x,t),\bar B(\x,t)) \bar B(\x,t) \ud \x \ud t.
\end{aligned}
\end{equation}
From \eqref{eq-20}, \eqref{eq-multline-1}, we obtain
\begin{equation}\label{eq-limsup-3}
\dsp\limsup_{m\to \infty}\dsp\int_\O (\Pi_{\disc_m}B_m(\x,t_m))^2 \ud \x
\leq \dsp\int_\O \bar B(\x,t_0)^2 \ud \x.
\end{equation}
Estimates \eqref{est-new} and \eqref{eq-est-dual}, and \cite[Theorem 4.19]{30} imply the weak convergence of $(\Pi_{\disc_m}B_m)_{m\in\NN}$ to $\bar B$ in $L^2(\O)$, it is indeed uniformly in $[0,T]$. This yields the weak convergence of $\Pi_{\disc_m}B_m(\cdot,s_m)$ to $\bar B(\cdot,t_0)$ in $L^2(\O)$. As a consequence of Estimate \eqref{eq-limsup-3}, this convergence of $\Pi_{\disc_m}B_m(\cdot,s_m)$ holds in the strong sense in $L^2(\O)$. With the continuity of $\bar B:[0,T] \to L^2(\O)$, we can apply \cite[Lemma C.13]{30} to conclude the strong convergence of $\Pi_{\disc_m}B_m$ in $L^\infty(0,T;L^2(\O))$ .    

\vskip .5pc
\noindent\textbf{Step 3: Convergence towards the solution to the continuous model.}\\
Recall that $A_m^{(0)}=I_{\disc_m}\bar A_{\rm ini}$, therefore the consistency shows that $\Pi_{\disc_m}A_m^{(0)}$ converges strongly to $A_{\rm ini}$ in $L^2(\O)$, as $m \to \infty$. Hence, $\bar A\in C([0,T];L^2(\O))\cap \mathbb K$ and $\bar A$ satisfies all conditions except the integral inequality imposed on the exact solution of Problem \eqref{pvi-obs-w1}. Let us now show that this integral relation holds. With Hypothesis \ref{assm-obs}, the dominated convergence theorem leads to $F(\Pi_{\disc_m}A_m,\Pi_{\disc_m}B_m) \to F(\bar A,\bar B)$ in $L^2(\O\times(0,T))$. The $L^2$--weak convergence of $\nabla_{\disc_m}A_m$ yields
\begin{equation}\label{ch5-liminf}
\dsp\int_0^T\int_\O {\bf D_A}\nabla\bar A\cdot \nabla\bar A \ud \x \ud t 
\leq \liminf_{m\to \infty}\dsp\int_0^T\int_\O {\bf D_A}\nabla_{\disc_m}A_m\cdot \nabla_{\disc_m} A_m \ud \x \ud t. 
\end{equation}

Fix $\kappa>0$ and let $\bar w_\kappa\in \mathbb L_\kappa$, where $\mathbb L_\kappa$ is defined by \eqref{ch5-dense-set}. Thanks to Lemma \ref{ch5-lem-intp}, there exists a sequence $(w_m)_{m\in\NN}$ such that $w_m \in \cK_{\disc_m}^{N_m+1}$ and $\Pi_{\disc_m}w_m \to \bar w_\kappa$ strongly in $L^2(\O\times(0,T))$ and $\nabla_{\disc_m}w_m \to \nabla\bar w_\kappa$ strongly in $L^2(\O\times(0,T))^d$. Setting $\varphi:=w_m$ as a generic function in \eqref{gs-pvi-obs1}, Inequality \eqref{ch5-liminf} implies that
\begin{align*}
\dsp\int_0^T\int_\O {\bf D_A}\nabla\bar A\cdot \nabla\bar A \ud \x \ud t 
\leq& \liminf_{m\to\infty}\bigg[\dsp\int_0^T\int_\O F(\Pi_{\disc_m}A_m,\Pi_{\disc_m}B_m) \Pi_{\disc_m}(A_m-w_m) \ud \x \ud t\\
&+\dsp\int_0^T\int_\O {\bf D_A}\nabla_{\disc_m}A_m\cdot \nabla_{\disc_m} w_m \ud \x \ud t\\
&-\dsp\int_0^T\int_\O\delta_{\disc_m}A_m \Pi_{\disc_m}(A_m-w_m) \ud \x \ud t\bigg].
\end{align*}
With the weak--strong convergences, pass to the limit in this relation to obtain, for all $w_\kappa\in \mathbb L_\kappa$, for all $\kappa>0$,
\begin{equation*}
\left.
\begin{array}{ll}
\dsp\int_0^T\int_\O \partial_t \bar A(\x,t)(u-w_\kappa)(\x,t) \ud \x \ud t +\dsp\int_0^T\int_\O {\bf D_A}\nabla\bar A \cdot \nabla(\bar A-w_\kappa)(\x,t)\ud \x \ud t\\
[1em]
\leq\dsp\int_0^T \int_\O F(\bar A(\x,t),\bar B(\x,t))(\bar A-w_\kappa)(\x,t)\ud \x \ud t.
\end{array}
\right.
\end{equation*}
By the density of the set $C_0^\infty(\overline\O)\cap \cK$ in $\cK$ proved in \cite{Density-ref}, every $\varphi\in \mathbb K$ can be expressed in a piecewise constant function in time $w_\kappa \in \mathbb L_\kappa$ such that $w_\kappa \to \varphi$ strongly in $L^2(0,T;H_0^1(\O))$ as $\kappa\to 0$ (note that $w_\kappa \leq \chi$ in $\O \times (0,T)$). Hence, \eqref{pvi-obs-w1} is verified.


Let us verify the integral equality \eqref{pvi-obs-w2}. Let $\psi$ be a generic function in $L^2(0,T;L^2(\O))$ and satisfies $\partial_t \psi \in L^2(\O\times(0,T))$ and $\psi(T,\cdot)=0$. Using the technique in \cite[Lemma 4.10]{30}, we can construct  $w_m=(w_m^{(n)})_{n=0,...,N_m} \in X_{\disc_m,0}^{N_m+1}$, such that $\Pi_{\disc_m}w_m \to \psi$ in $L^2(0,T;L^2(\O))$ and $\delta_{\disc_m}w_m \to \partial_t \psi$ strongly in $L^2(\O\times(0,T))$. Take $\psi=\delta t_m ^{(n+\frac{1}{2})}w_m^{(n)}$ as a generic function in \eqref{gs-pvi-obs2} and sum on $n\in[0,N_m-1]$ to get
\begin{equation}\label{eq-part1}
\begin{aligned}
\dsp\sum_{n=0}^{N_m-1}\dsp\int_\O [ \Pi_{\disc_m}B_m^{(n+1)}(\x)&-\Pi_{\disc_m}B^{(n)}(\x) ]\Pi_{\disc_m}w_m^{(n)}(\x) \ud x\\
&\quad+\dsp\int_0^T\dsp\int_\O {\bf D_B}(\x)\nabla_{\disc_m}B_m(\x,t) \cdot \nabla_{\disc_m}w_m(\x,t) \ud \x \ud t\\
&=\dsp\int_0^T\dsp\int_\O G(\Pi_{\disc_m}A(\x,t),\Pi_{\disc_m}B(\x,t))\Pi_{\disc_m}w_m(\x,t) \ud x \ud t.
\end{aligned}
\end{equation}
Applying the formula \cite[Eq. (D.15)]{30} to the right hand side yeilds, thanks to $w^{(N)}=0$ 
\begin{equation*}
\begin{aligned}
-\dsp\int_0^T \dsp\int_\O \Pi_{\disc_m}B_m(\x,t)&\delta_{\disc_m}w_m(\x,t) \ud \x \ud t 
-\int_\O \Pi_{\disc_m}B_m^{(0)}(\x)\Pi_{\disc_m}w_m^{(0)}(\x) \ud \x\\
&+\dsp\int_0^T\dsp\int_\O {\bf D_B}(\x)\nabla_{\disc_m}B_m(\x,t) \cdot \nabla_{\disc_m}w_m(\x,t) \ud \x \ud t\\
&\quad=\dsp\int_0^T\dsp\int_\O G(\Pi_{\disc_m}A(\x,t),\Pi_{\disc_m}B(\x,t))\Pi_{\disc_m}w_m(\x,t) \ud x \ud t.
\end{aligned}
\end{equation*}
By the consistency, we see $\Pi_{\disc_m}B_m^{(0)}=\Pi_{\disc_m}J_{\disc_m}B_{\rm ini} \to B_{\rm ini}$ in $L^2(\O)$. This with passing to the limit $m\to \infty$ implies, for all $\psi$ in $L^2(0,T;H^1(\O))$,
\begin{equation*}
\begin{aligned}
&-\dsp\int_0^T\dsp\int_\O \partial_t \psi(\x,t)\bar B(\x,t) \ud \x \ud t
+\dsp\int_0^T\int_\O {\bf D_B}\nabla\bar B \cdot \nabla\psi(\x,t)\ud \x \ud t\\
&=\dsp\int_0^T \int_\O G(\bar A,\bar B)\psi(\x,t)\ud \x \ud t.
\end{aligned}
\end{equation*}
Since $C^\infty([0,T];H^1(\O))$ is dense in $L^2(0,T;H^1(\O))$, integrating by part  shows that the above equality can be expressed in the sense of distributions, which is equivalent to \eqref{pvi-obs-w2}.


\medskip
\noindent\textbf{Step 4: Proof of the strong convergence of $\nabla_{\disc_m}A_m$ and $\nabla_{\disc_m}B_m$.}\\
From the weak--strong  convergences, we have, for all $\bar w_\kappa \in \mathbb L_\kappa$,
\begin{align*}
\dsp\limsup_{m\to \infty}\dsp\int_0^T\int_\O {\bf D_A}\nabla_{\disc_m}A_m\cdot \nabla_{\disc_m} A_m \ud \x \ud t
\leq \dsp\int_0^T\int_\O F(\bar A,\bar B) (\bar A-\bar w_\kappa) \ud \x \ud t\\
+\dsp\int_0^T\int_\O {\bf D_A}\nabla\bar A\cdot \nabla\bar w_\kappa \ud \x \ud t
-\dsp\int_0^T\int_\O\partial_t \bar A(\bar A-\bar w_\kappa)\ud \x \ud t. 
\end{align*}
Thanks to the density results, for any $\varphi\in \mathbb K$, we can find $(\bar w_\kappa)_{\kappa>0}$ that converges to $\varphi$ in $L^2(0,T;H^1(\O))$, as $\kappa\to 0$. Therefore, we infer, for all $\varphi\in \mathbb K$,
\begin{align*}
\dsp\limsup_{m\to \infty}\dsp\int_0^T\int_\O {\bf D_A}\nabla_{\disc_m}A_m\cdot \nabla_{\disc_m} A_m \ud \x \ud t
\leq \dsp\int_0^T\int_\O F(\bar A,\bar B) (\bar A-\varphi) \ud \x \ud t\\
+\dsp\int_0^T\int_\O {\bf D_A}\nabla\bar A\cdot \nabla \varphi \ud \x \ud t
-\dsp\int_0^T\int_\O\partial_t \bar A(\bar A-\varphi)\ud \x \ud t.
\end{align*}
With taking $\varphi=\bar A$, the above relation yields
\begin{equation*}
\limsup_{m\to \infty}\dsp\int_0^T\int_\O {\bf D_A}\nabla_{\disc_m}A_m\cdot \nabla_{\disc_m} A_m \ud \x \ud t 
\leq\dsp\int_0^T\int_\O {\bf D_A}\nabla\bar A\cdot \nabla\bar A \ud \x \ud t. 
\end{equation*}
Together with \eqref{ch5-liminf}, we obtain 
\begin{equation*}\label{ch5-eq-lim}
\lim_{m\to \infty}\dsp\int_0^T\int_\O {\bf D_A}\nabla_{\disc_m}A_m\cdot \nabla_{\disc_m} A_m \ud \x \ud t 
=\dsp\int_0^T\int_\O {\bf D_A}\nabla\bar A\cdot \nabla\bar A \ud \x \ud t,
\end{equation*}
which implies 
\begin{align*}
0&\leq d_1\limsup_{m\to \infty}\dsp\int_0^T\int_\O|\nabla\bar A-\nabla_{\disc_m} A_m|^2 \ud \x \ud t \\
&\leq \limsup_{m\to \infty}\bigg[ \dsp\int_0^T\int_\O {\bf D_A}\nabla\bar A \cdot \nabla\bar A + \dsp\int_0^T\int_\O {\bf D_A}\nabla_{\disc_m}A_m \cdot \nabla_{\disc_m}A_m \ud \x \ud t \\
&-2\dsp\int_0^T\int_\O {\bf D_A}\nabla\bar A \cdot \nabla_{\disc_m}A_m \ud \x \ud t \bigg]\\
&=0,
\end{align*}
which shows that $\nabla_{\disc_m}A_m \to \nabla\bar A$ strongly in $L^2(\O\times(0,T))^d$. To show the strong convergence of $\nabla_{\disc_m}B_m$, we begin with writing
\begin{equation}\label{es-eq}
\begin{aligned}
\dsp\int_0^T\dsp\int_\O (\nabla_{\disc_m}B_m(\x,t)&- \nabla\bar B(\x,t))\cdot(\nabla_{\disc_m}B_m(\x,t)- \nabla\bar B(\x,t)) \ud \x \ud t\\
&=\dsp\int_0^T\dsp\int_\O \nabla_{\disc_m}B_m(\x,t) \cdot \nabla_{\disc_m}B_m(\x,t) \ud \x \ud t\\
&\quad-\dsp\int_0^T\dsp\int_\O \nabla_{\disc_m}B_m(\x,t) \cdot \nabla\bar B(\x,t) \ud \x \ud t\\
&\quad-\dsp\int_0^T\dsp\int_\O \nabla\bar B(\x,t) \cdot (\nabla_{\disc_m}B_m(\x,t)-\nabla\bar B(\x,t)) \ud \x \ud t.
\end{aligned}
\end{equation}
With setting $\psi:=B_m$ in \eqref{gs-pvi-obs2} and $\psi = \bar B$ in \eqref{pvi-obs-w2}, and taking ${\bf D_B}(\x)={\bf Id}$, passing to the limit superior yields 
\begin{equation*}
\begin{aligned}
\dsp\limsup_{m\to \infty}\dsp\int_0^T\dsp\int_\O \nabla_{\disc_m}&B_m(\x,t)\cdot \nabla_{\disc_m}B_m(\x,t) \ud x \ud t\\
&=\dsp\int_0^T\dsp\int_\O G(\bar A,\bar B)\bar B(\x,t) \ud \x \ud t
-\dsp\int_0^T\dsp\int_\O \partial_t \bar B(\x,t) \bar B(\x,t) \ud \x \ud t
\\
&=\dsp\int_0^T\dsp\int_\O \nabla\bar B(\x,t)\cdot \nabla\bar B(\x,t)\ud \x \ud t.
\end{aligned}
\end{equation*}
Pass to the limit in \eqref{es-eq} and use the above inequality to reach the desired convergence.
\end{proof}

\section{An example of schemes covered by the analysis}\label{sec-ex}
Many numerical schemes fit in the analysis provided in this work. We construct here the non conforming $\mathbb P 1$ finite element scheme for our model. Let $\polyd = (\mesh,\edges)$ be polytopal mesh of $\O$ defined in \cite[Definition 7.2]{30}, in which $\mesh$ and $\edges$ consists of cells $K$ and edges $\edge$, respectively. The elements of gradient discretisation $\disc$ associated with the non conforming $\mathbb P1$ finite element scheme are: 
\begin{itemize}
\item 
$
X_{\disc,0}=\{ w=(w_\edge)_{\edge\in\edges}\;:\; w_\edge \in \RR \mbox{ and } w_\edge=0 \mbox{ for all } \edge \in \edgesext   \}.
$
\item For all $w\in X_{\disc,0}$ and for all $K\in \mesh$, for a.e. $\x\in K$,
\[
\Pi_\disc w(\x) =\dsp\sum_{\edge\in\edges_K}w_\edge e_K^\edge(\x),
\]
where $e_K^\edge$ is a basis function. 
\item For all $w\in X_{\disc,0}$ and for all $K\in \mesh$, for a.e. $\x\in K$,
\[
(\nabla_\disc w)_{|K}=\nabla[(\Pi_\disc w)_{|K}]=\dsp\sum_{\edge\in\edges_K}w_\edge \nabla e_K^\edge
\]

\item The approximate obstacle $\chi_\disc$ is defined by
\[
\chi_\disc:=\fint_\edge \chi(\x) \ud \x.
\]
\item For all $\omega\in W^{2,\infty}(\O)$, we can construct the interpolants $I_\disc\omega=J_\disc\omega=(z_\edge)_{\edge\in\edges}$ with $z_\edge=\omega(\centeredge)$.
\end{itemize}

Substituting these elements in Scheme \eqref{gs-pvi} yields the non conforming $\mathbb P1$ finite element scheme for Problem \eqref{pvi-obs-weak} and its convergence is therefor obtained from Theorem \ref{theorem-conver-rm}. \cite{new-52} shows that $\disc$ given here satisfies the three properties, coercivity, limit conformity and compactness. Let us discuss the consistency property in the sense of Definition \ref{def:cons-rm}. It is shown in \cite{new-52}, that $\widetilde S_{\disc_m}(\psi) \to 0$, for all $\psi \in H_0^1(\O)$, which verifies the second item. Similarly, we can prove that $S_{\disc_m}(\varphi) \to 0$ for all $\varphi \in C^2(\overline\O)\cap\cK$. For any $\varphi$, let $\omega=(\omega_\edge)_{\edge\in\edges} \in X_{\disc,0}$ be the interpolant such that $\omega_\edge=\varphi(\centeredge)$, for all $\edge\in\edges$. We clearly deduce $\Pi_\disc \varphi \leq \chi_\disc$ in $\O$. By the density results established in \cite{Density-ref}, we see that the first item is fulfilled.
 
Let $\varphi_m=(\omega_\edge)_{\edge\in\edges_{m}}\in \cK_{\disc_m}$ and $\psi_m=(w_\edge)_{\edge\in\edges_{m}}\in X_{\disc_m,0}$ be the interpolants such that $\varphi_m=I_{\disc_m}A_{\rm ini}$ and $\psi_m=J_{\disc_m}B_{\rm ini}$. \cite[Estimate (B.11), in Lemma B.7]{30} with $p=2$ shows that there exists $C_5,\;C_6 >0$ not depending on $m$ such that, 
\[
\| \bar A_{\rm ini}- \Pi_\disc I_{\disc_m} A_{\rm ini} \|_{L^2(\O)}^2 \leq C_5^2 h_{\mesh_m}^2 \| \nabla A_{\rm ini}\|_{L^2(\O)}^2, \mbox{ and}
\]
\[
\| \bar B_{\rm ini}- \Pi_\disc J_{\disc_m}B_{\rm ini} \|_{L^2(\O)}^2 \leq C_6^2 h_{\mesh_m}^2 \| \nabla B_{\rm ini}\|_{L^2(\O)}^2.
\]
Pass to the limit to see the right hand sides tend to $0$ (thanks to the classical regularity Hypothesis on $A_{\ini}$ and $B_{\ini}$), and therefore the third and fourth items of the consistency property are verified. Finally, it is shown in \cite[Proof of Theorem 12.12]{30} that, for $\varphi\in W^{1,p}(\O)$, we can construct a function $\omega_m=(\omega_\edge)_{\edge\in\edges_{m}}\in \cK_{\disc_m}$, such that there is $C_7 >0$ not depending on $m$ such that
\[
\| \nabla_{\disc_m} \omega_m \|_{L^p(\O)^d} \leq C_7 \| \nabla \varphi \|_{L^p(\O)^d}.
\]
Apply this estimate (with $p=2$) to $\varphi=A_{\rm ini}$ and $\omega_m=I_{\disc_m}A_{\rm ini}$ to deduce $\| \nabla_{\disc_m}I_{\disc_m}A_{\rm ini} \|_{L^2(\O)^d}$ is bounded.

The non conforming $\mathbb P1$ finite element method for the problem \eqref{pvi-obs} is, for all $n=0,...,N-1$, the following holds:
\[
\begin{aligned}
\Big(\frac{|\edge|}{\delta t^{(n+\frac{1}{2})}}( A_\edge^{(n+1)}&-A_\edge^{(n)})
+\sum_{\sigma \in \edges_K}|\edge|A_\edge^{(n+1)}{\bf n}_{K,\edge}- |K|F(A_\edge^{(n+1)},B_\edge^{(n+1)})\Big)(A_\edge^{(n+1)}-\chi_\edge)=0,\\
&\mbox{ for all $K\in\mesh$ and for all $\edge \in \edges_K$},
\end{aligned}
\]
\[
\frac{|\edge|}{\delta t^{(n+\frac{1}{2})}}( A_\edge^{(n+1)}-A_\edge^{(n)} )+\sum_{\sigma \in \edges_K}|\edge|A_\edge^{(n+1)}{\bf n}_{K,\edge}\geq |\edge|F(A_\edge^{(n+1)},B_\edge^{(n+1)})\quad \mbox{ for all } \edge \in \edges,
\]
\[
A_\edge^{(n+1)} \geq \chi_\edge \quad \mbox{ for all } \edge \in \edges,
\]
\[
\begin{aligned}
\frac{|\edge|}{\delta t^{(n+\frac{1}{2})}}( B_\edge^{(n+1)}-B_\edge^{(n)})&+\sum_{\sigma \in \edges_K}|\edge|B_\edge^{(n+1)}{\bf n}_{K,\edge}= |\edge|G(A_\edge^{(n+1)},B_\edge^{(n+1)}),\\
&\mbox{ for all $K\in\mesh$ and for all $\edge \in \edges_K$},
\end{aligned}
\]
\[
A_\sigma^{(n+1)}=B_\sigma^{(n+1)}=0 \quad \mbox{ for all } \sigma \in \edges_{\rm ext},
\]
\[
(A^{(0)},B^{(0)})=\left(A_{\rm ini}(x_\edge,0),B_{\rm ini}(x_\edge,0)\right)\quad \mbox{ for all } \edge \in \edges.
\]



\bibliographystyle{siam}
\bibliography{pvi-ref}

\end{document}